\documentclass[11pt]{amsart}

\usepackage{amsfonts,amssymb,mathrsfs}
\usepackage{amsmath}
\usepackage{amssymb}

\theoremstyle{plain}
\newtheorem{definition}{Definition}
\newtheorem{theorem}[definition]{Theorem}
\newtheorem{lemma}[definition]{Lemma}

\newtheorem{prop}[definition]{Proposition}

\newtheorem*{ack}{Acknowledgement}

\newcommand{\PP}{\mathcal P}

\newcommand{\Z}{\mathbb{Z}}
\newcommand{\N}{\mathbb{N}}
\newcommand{\F}{\mathcal{F}}

\newcommand{\nhat }[1]{\{1,\ldots,#1\}}

\renewcommand{\subset}{\subseteq}

\newcommand{\pb}[1]{^{(#1)}}
\newcommand{\pob}[1]{_{#1}}

\newcommand\Cal{\mathcal}

\newcommand\na{\mathbb N}

\newcommand\im{invariant measure}

\newcommand\oo{\overline O}

\title{Solvability of Rado systems in D-sets}
\author{M. Beiglb\"ock, V. Bergelson, T. Downarowicz, A. Fish}

\begin{document}

\maketitle\

\begin{center}
This paper is dedicated to Neil Hindman on the occasion of his 65th birthday.
\end{center}

\begin{abstract}
Rado's Theorem characterizes the systems of homogenous linear equations having the property that for any finite partition of the positive integers one cell contains a solution to these equations. Furstenberg and Weiss proved that solutions to those systems can in fact be found in  every central set. (Since one cell of any finite partition is central, this generalizes Rado's Theorem.) We show that the same holds true for the larger class of $D$-sets. Moreover we will see that the conclusion of Furstenberg's Central Sets Theorem is true for all sets  in this class.
\end{abstract}

\section{Introduction}

Schur's Theorem (\cite{Schu16}) states that for any finite partition of the positive integers one cell contains solutions to the equation $x_1+x_2=x_3$.

Another classical result of partition Ramsey theory is van der Waerden's Theorem \cite{Waer27} which states that arithmetic progressions of arbitrary finite length can be found in one cell of any finite partition. This follows from the fact that a solution to the equations $x_1=x_3-x_2=\ldots=x_n-x_{n-1}$ 
can always be found in one cell.\footnote{This somewhat stronger result which asserts that the arithmetic progression and the increment  can be forced to lie in the same cell is due to Brauer \cite{Brau28}.}

Both statements are special instances of Rado's Theorem which provides necessary and sufficient conditions  for the system $A(x_1,\ldots, x_q)^T=0$, $A\in \Z^{p\times q}$ 
  to be partition regular in the sense that for every finite partition of the positive integers one cell  contains $x_1, \ldots, x_q$ satisfying
$A(x_1,\ldots, x_q)^T=0$. Each such system of linear equations is called a \emph{Rado system}.

\begin{theorem}[Rado's Theorem \cite{Rado33}]
A system of linear equations of the form $A(x_1,\ldots, x_q)^T=0, A=(a_{ij})\in \Z^{p\times q}$ is a Rado system iff the index set \( \{1,2,\ldots,q\}\) can be divided
into disjoint subsets \( I \pob 1,I \pob 2,\ldots,I \pob l \) and for all $r\in \nhat l, j \in I \pob 1 \cup \ldots \cup I \pob r $ rational
numbers $ c_j^r$ may be found  such that the following relations are
satisfied:
$$
\sum_{j \in I \pob 1} a_{ij} = 0$$
$$\sum_{j \in I \pob 2} a_{ij} = \sum_{j \in I \pob 1} c_j^1 a_{ij}$$
$$\ldots$$
$$\sum_{j \in I \pob l} a_{ij} = \sum_{j \in I \pob 1 \cup I \pob 2 \cup \ldots \cup
I \pob {l-1}} c_j^{l-1} a_{ij}.$$
\end{theorem}

We want to mention a corollary\footnote{This was proved independently (but much later then Rado's Theorem) by  Folkman (unpublished) and Sanders \cite{Sand69}.} of Rado's Theorem extending Schur's Theorem.
It is possible to find arbitrarily many numbers  $x_1, \ldots, x_n$, together with all finite sums $x_{k_1}+ \ldots+x_{k_l}, k_1< \ldots<k_l\leq n$ in one cell of any finite partition. 
Only some forty years after the publication of Rado's result, Hindman  (\cite[Theorem 3.1]{Hind74}) established that one can actually find an {infinite} sequence together with all finite sums from its elements in one cell. (See \cite{DHLL95, HiSt00, HiLS03} for more information on infinite partiton regular systems of equations.)

Sometimes a deeper understanding of results in Partition Ramsey Theory is achieved by finding the proper notion of \emph{largeness} which guarantees that one cell of a finite partition contains rich combinatorial structure.  The Theorems of van der Waerden and Szemer\'edi provide an example of this principle: While the first one states that every finite partition of the integers has one cell which contains arbitrarily long arithmetic progressions, the latter reveals that those can be found in every set $S$ of positive upper Banach density $d^*(S)= \overline\lim_{(m-n)\to\infty} |S\cap \{n,\ldots, m\}|/(m-n+1)$. Clearly, at least one cell of each finite partition of $\N$ has positive upper Banach density and thus van der Waerden's Theorem is a corollary of Szemer\'edi's Theorem.

Furstenberg and Weiss improved Rado's result by showing that solutions to Rado systems can be found in every \emph{central set}.\footnote{The Theorem in this form was spelled out in \cite[Theorem 8.22]{Furs81} but can also be deduced from \cite[Theorem 4.4]{FuWe78} which was published before the introduction of central sets.}
One can easily give a short description of central sets using ultrafilters on $\N$, but since we want to postpone dealing with these somewhat esoteric objects to Section 3, we content ourselves with  Furstenberg's original definition (\cite[Definition 8.3]{Furs81}).

A set $S\subseteq \N$ is central iff there exist a dynamical system $(X,T)$ (i.e.\ a compact metric space $X$ and a continuous transformation $T$ of $X$), a point $x\in X$, a uniformly recurrent $y$ in the orbit closure of $x$, and an open neighborhood $U$ of $y$ such that
$$S=\{n\in \N:T^n x\in U\}.$$ ($y\in X$ is called uniformly recurrent if for each neighborhood $U$ of $y$  the set $\{n\in \N:T^n y \in U\}$ is syndetic, i.e.\ has bounded gaps.) 



One can then prove (and this is in fact apparent from the ultrafilter description given in Section \ref{ultrafilterSection}) that one cell of each finite partition of the positive integers is central, that every set containing a central set is central itself and that central sets remain central after removing finitely many points.

Central sets have positive upper Banach density. In fact, if $S$ is central, it posseses the strictly stronger property that there exists $k\in \N$ such that $S\cup (S -1) \cup \ldots \cup (S-k)$ contains arbitrarily long intervals, i.e.\ $S$ is \emph{piecewise syndetic}.

While it is merely an exercise to derive van der Waerden's Theorem from the fact that every piecewise syndetic set contains arbitrarily long arithmetic progressions, Szemer\'edi's Theorem which guarantees the existence of arbitrarily long arithmetic progressions in sets of positive upper Banach density is highly nontrivial. Analogously, one might search for a class of not necessarily piecewise syndetic sets which contain solutions to Rado systems. Clearly positive upper Banach density is not the appropriate notion (for instance the set of all odd numbers contains no configuration of the form $x_1,x_2, x_1+x_2$).
In fact it is possible to find for each $\varepsilon>0$ a set $S$ of density bigger than $1-\varepsilon$ such that for no $t\in \Z$, $S-t$ contains solutions to all Rado systems.\footnote{Following Ernst Strauss (see \cite[Theorem 2.20]{BBHS06}) one can construct a set $S$ with density arbitrarily close to $1$ such that there doesn't exist $t\in\Z$ such that $(S-t)\cap \N n\neq \emptyset$ for every $n\in\N$.
Given positive integers $ x_1,\ldots,x_m, m=n^2$ there exist $i_1<\ldots<i_k\leq m$ such that $x_{i_1}+\ldots+x_{i_k}\in n\N$.
 Hence if $(S-t)\cap \N n =\emptyset$, $S-t$ cannot contain positive integers  $x_1, \ldots, x_{m}$ and all finite sums from these numbers. In particular, no shifted copy of $S$ contains solutions to all Rado systems.} (In contrast to this it is always possible to shift a piecewise syndetic set such that it becomes central (see \cite[Theorem 4.40]{HiSt98}) and then contains solutions to all Rado systems.)

In this paper we prove that solutions of Rado systems are contained in any member of a class of sets which is larger than the class of central sets. This class is comprised of $D$-sets defined in \cite{BeDo08}. The main distinction of $D$-sets from the class of central sets is that in the definition of central sets instead of a uniformly recurrent point, one considers an \emph{essentially recurrent point} $y$, meaning that the set $\{n\in\N  : T^n y \in U\}$ has positive upper Banach density for every neighborhood $U$ of $y$. Note that since every syndetic set has positive upper density, every uniformly recurrent point is an essentially recurrent point.

A set $S\subseteq \N$ is a \emph{$D$-set} iff there exist a dynamical system $(X,T)$ (i.e.\ a compact metric space $X$ and a continuous transformation $T$ of $X$), a pair of points $x,y\in X$  where $y$ is essentially recurrent, and such that $(y,y)$ belongs to the orbit closure of $(x,y)$ in the product system $(X\times X, T\times T)$, and an open neighborhood $U$ of $(y,y)$ such that
$$S=\{n\in \N:(T^n x, T^n y)\in U\}.$$

By \cite[Theorem 2.3]{BeDo08} a set $S\subseteq \N$ is central iff it satisfies the above definition with the twist that $y$ is not just essentially recurrent, but a uniformly recurrent point. Hence every central set is a $D$-set. 
 Similarly to central sets, the family of $D$-sets is closed under forming supersets and every $D$-set has positive upper Banach density. But $D$-sets  don't need to be piecewise syndetic. 
(See \cite{BeMc08}.) In particular the class of $D$-sets is strictly larger than the class of central sets.

So our main result is:

\begin{theorem}\label{mainthm}
Rado systems  are solvable in $D$-sets.
\end{theorem}

We will give two proofs of Theorem \ref{mainthm}. The first one, given in Section
\ref{dynamicSection} is formulated in the language of topological dynamics, while the second  one, presented in Section \ref{ultrafilterSection}  makes use of the algebraic structure on the set of ultrafilters on $\N$. This second proof actually establishes that Furstenberg's Central Sets Theorem (\cite[Proposition 8.21]{Furs81} can be extended to $D$-sets.
In Section \ref{preliminariesSection} we collect some tools which will be used in both proofs of Theorem \ref{mainthm}.  

Note that in \cite{BeDo08} $D$-sets are defined as subsets of the group $\Z$ and also the transformations considered there are invertible. However it is more traditional to work with subsets of the positive integers when the focus of interest lies on combinatorial applications. The proofs of the statements  in \cite{BeDo08} work in this modified setting without any significant changes and the connection between $D$-sets in $\N$ and $D$-sets in $\Z$ is rather natural: Every $D$-set in $\N$ is a $D$-set in $\Z$ and $S\subset \Z$ is a $D$-set iff $S\cap\N$ or $(-S)\cap \N$ is a $D$-set in $\N$. (The analogous statement holds true for central sets.) 

One might wonder whether any set satisfying the conclusion of the Central Sets Theorem
must have positive Banach density.  It is shown in \cite{Hind08}
that this is not the case.

\section{Preliminaries}\label{preliminariesSection} 

The following concept is due to Deuber (\cite{Deub73}, see also \cite[Chapter 15]{HiSt98}.):
  Given positive integers $m,p,c$, the \emph{\( (m,p,c) \)-system} generated by the $(m+1)$-tuple $(s^{(0)},\ldots,s^{(m)})$ is the
following array of numbers:
\[c s\pb0,\]
\[c s\pb1+i \pob 0s\pb0, \,\, |i \pob 0| \leq p,\]
\[\cdots\]
\[c s\pb m + i \pob {m-1}s\pb{m-1} + \ldots + i \pob 0 s\pb0, \,\, |i \pob {m-1}|,\ldots,|i \pob 0| \leq p.\]
Deuber (\cite[Satz 2.1]{Deub73}) proved that every 
 Rado-system is solvable within a set $S$  of positive integers iff for any triple $(m,p,c)$ of positive integers   \( S \) contains an \((m,p,c)\)-system. Thus for our purposes it is sufficient to prove the following result.

\begin{prop}\label{viaDeuber}
Let $S$ be a $D$-set and $m,p,c$ positive integers. Then $S$ contains an 
$(m,p,c)$-system.
\end{prop}

Since we are going to prove the existence of structures extending arithmetic progressions in sets which need not be piecewise syndetic, it is no surprise that we will employ some version of Szemer\'edi's Theorem. In fact we shall use the Furstenberg and Katznelsons's deep  \emph{multiple $IP$ recurrence Theorem} and its combinatorial corollary, the \emph{$IP$ Szemer\'edi Theorem}.

To formulate these theorems we introduce some notation: By $\F$ we denote the set of all finite nonempty sets of positive integers. For $\alpha, \beta \in \F$, we write $\alpha <\beta $ iff $\max \alpha< \min \beta$.  
Given a sequence $s_ 1 , s_ 2, \ldots$  in $\Z$ or $\Z^m$ and $\alpha=\{k_1, \ldots , k_l\}\in \F, k_1<\ldots<k_l$, we let 
$s_\alpha=s_{k_1}+\ldots+s_ {k_l}$ and call the family $(s_\alpha)_{\alpha\in\F}$ an \emph{$IP$-system}. 
Similarly, for a  sequence $T_ 1, T _ 2 ,\ldots $ of commuting transformations of a space, we assign to $\alpha$ the transformation $T_\alpha = T_ {k_1}\circ\ldots \circ T_ {k_l}$ and call $(T_\alpha)_{\alpha\in\F}$ an \emph{$IP$-system of transformations}.

\begin{theorem}\label{IPmrt}(multiple $IP$-recurrence theorem, see \cite[Theorem A]{FuKa85})
Let \( (X,\mathcal{B},\mu) \) be a probability measure space. Let 
$(T\pb1_\alpha)_{\alpha\in \F}, \ldots, (T\pb p_\alpha)_{\alpha\in \F}$ be commuting $IP$-systems of 
transformations which preserve  \( \mu \). Then for
every \( A \in \mathcal{B} \) with \( \mu(A) > 0\) there exists \(
\alpha \in \F \) such that
\[
\mu\big(A\cap \big(T_{\alpha}\pb1\big)^{-1} A \cap \ldots \cap \big(T_{\alpha}\pb{p}\big)^{-1}A\big) > 0.
\]
\end{theorem}

We plan to apply Theorem \ref{IPmrt} in the dynamical proof of Proposition \ref{viaDeuber}. The link to $D$-sets will be established using the following result:

\begin{theorem}(\cite[Theorem 2.6]{BeDo08})
\label{berg-down} Let \( (X,T) \) be a dynamical system, let \( y \in X \) be an essentially recurrent point  and $U$ a neighborhood of $y$. Then  
there exists a probability Borel measure \( \mu \) on \( X \)
preserved by the action \( T \) such that \( \mu(U) > 0 \).
\end{theorem}

 In the ultrafilter proof of Proposition \ref{viaDeuber} we use the above mentioned combinatorial  corollary of Theorem \ref{IPmrt}.

\begin{theorem}\label{combSze}($IP$ Szemer\'edi Theorem) 
Assume that $S\subset \N$ has posiive upper Banach density and let $(s\pb1_\alpha)_{\alpha \in \F}, \ldots, (s\pb p_\alpha)_{\alpha \in \F}$ be $IP$-systems of integers. Then there exist  $\alpha\in \F$ and $a\in S$ such that $a +s\pb1_\alpha, \ldots, a+ s\pb p_\alpha \in S$.
\end{theorem}

We'll also need the following lemma on $IP$-systems. (The proof is left as an exercise.)
\begin{lemma}\label{divisible}
Let $(s_\alpha)_{\alpha\in \F}$ be an $IP$-system of integers and let $c\in\N$. There exist $\alpha_1<\alpha_2<\ldots$ in $\F$ such that for every $n\in \N, s_{\alpha_n}$ is divisible by $c$.  
\end{lemma}

\section{A proof via Topological Dynamics}\label{dynamicSection}

The proof of Proposition \ref{viaDeuber} is more transparent for $c=1$. Therefore we will first restrict ourselves to this special case and make some remarks on what needs to be changed to achieve the result in full generality later.

For the rest of this section, fix a dynamical system $(X,T)$ and $x,y\in X$ such that  $y$ is essentially recurrent, and such that $(y,y)$ belongs to the orbit closure of $(x,y)$ in the product system $(X\times X, T\times T)$.
Given an open neighboorhood $U$ of $(y,y)$, we let  $S_U=\{n\in \N:(T^n x, T^n y)\in U\}.$ In this setting Proposition \ref{viaDeuber} (for $c=1$) translates to:
\begin{prop}\label{smallversion}
     Let $m,p\in\N$ and let $U$ be an open neighborhood of $(x,y)$. Then $S_U$ contains an $(m,p,1)$-system.  
\end{prop}

The proof of Proposition \ref{smallversion} is based on the following Lemma:

\begin{lemma}\label{onetomany} Fix some $p\in\na$. If $m\ge 0$ is such that for every $U\ni (y,y)$, $S_U$ contains a $(m,p,1)$-system, then for every $U$,
$S_{U}$ contains a family of
$(m,p,1)$-systems such that their generating $(m+1)$-tuples
$\big(s_{\alpha}\pb0,\dots,s_{\alpha}\pb m\big)_{\alpha\in\F}$  form an $IP$-system in $\N^{m+1}$.
\end{lemma}

\begin{proof} Clearly, it suffices to consider {\it symmetric} sets $U$,
i.e.\ sets of the form $V\times V$, where $V$ is an open set containing
$y$.

Fix $m\ge 0$ for which the assumption holds. Fix a symmetric open
set $U_1\ni (y,y)$. We know that $S_{U_1}$ contains a $(m,p,1)$-system $D_1$ generated by some $(m+1)$-tuple $\big(s_1\pb 0,\dots,s_1\pb m\big)$. In particular, the set
$$
U'_2 = \bigcap_{n\in D_1} (T\times T)^{-n}(U_1)
$$
contains $(x,y)$ and, since $U_1$ was symmetric, also $(y,y)$. Let now
$U_2$ be a symmetric neighborhood of $(y,y)$ contained in $U_1\cap U'_2$. By
assumption, the set $S_{U_2}$ also contains a $(m,p,1)$-system
$D_2$ generated by $\big(s_2\pb 0,\dots,s_2\pb m\big)$. Clearly $S_{U_2}\subset S_{U_1}$, so $S_{U_1}$ contains both $D_1$ and $D_2$. Moreover, it
contains the algebraic sum of $D_1$ and $D_2$. In
particular, it contains the Deuber system $D_{1,2}$ generated by
$\big(s_1\pb0+s_2\pb0,\ldots, s_1\pb m+s_2\pb m\big)$.
Continuing by an obvious induction we construct a family of Deuber
systems as in the assertion. \end{proof}

\begin{proof}[Proof of Proposition \ref{smallversion}.] We proceed by induction on
$m$.

For $m=0$ the $(m,p,1)$-system reduces to a single number
$s_0$, such that $(T\times T)^{s_0}(x,y)\in U$. The set of such
numbers $s_0$ is nonempty for every open $U\ni(y,y)$, because
$(y,y)\in\oo(x,y)$.

Suppose the assertion holds for some $m$ and all open sets $U\ni(y,y)$.
Fix $U$. By Lemma \ref{onetomany}, $S_U$ contains many such systems indexed by
$\alpha\in \F$, where the generating $(m+1)$-tuples $\big(s_{\alpha}\pb0,\dots,s_{\alpha}\pb m\big)$ form an $IP$-system in $\N^{m+1}$. Then for any
fixed integers $i\pob 0,\dots,i\pob m$ the numbers
$i\pob 0s\pb 0_\alpha+\dots+i\pob m s\pb m_\alpha$ (with varying $\alpha$) form an
$IP$-system and
$\big(R^{i\pob 0s\pb 0_\alpha+\ldots+i\pob m s\pb m_\alpha}\big)_{\alpha\in\Cal F}$ is an $IP$-system of transformations for any given transformation $R$. Let $(i\pob 0,\dots,i\pob m)$ range over the integers in $[-p,p]^{m+1}$ and consider the following $(2p+1)^{m+1}$ commuting $IP$-systems of transformations on $X\times X$: 
$$T_\alpha^{(i \pob 0,\dots,i \pob m)} = (T\times T)^{i\pob 0s\pb 0_\alpha+\dots+i\pob m s\pb m_\alpha}.
$$
Apply Theorem \ref{berg-down} to $(y,y)\in \oo (y,y) (\subset (X\times X, T\times T))$ to get a $T\times T$-\im\ $\mu$ which assigns positive measure to $U$. Since $\mu$ is preserved by all
the above transformations, Theorem \ref{IPmrt}
asserts that there exists an $\alpha\in\Cal F$ such that
$$
V=\bigcap_{(i \pob 0,\dots,i \pob m)\in [-p,p]^{m+1}}
\big(T_\alpha^{(i \pob 0,\dots,i \pob m)}\big)^{-1}U 
$$
has positive measure. Since $\mu$ is supported by $\oo(y,y)\subset
\oo(x,y)$, these facts imply that there exists an integer $s$, such
that $(T\times T)^s(x,y) \in V$. This, in turn, implies that the
numbers $s + i \pob 0s_\alpha\pb0+\dots+i \pob m s_\alpha\pb m$ belong to $S_U$ for
all $(i \pob 0,\ldots, i \pob m)\in [-p,p]^{m+1}$. Because $S_U$ already
contains the $(m,p,1)$-system generated by the $(m+1)$-tuple
$\big(s_{\alpha}\pb0,\ldots,s_{\alpha}\pb m\big)$, we have proved that $S_U$ also
contains the $(m+1,p,1)$-system generated by the $(m+2)$-tuple
$\big(s_{\alpha}\pb0,\ldots,s_{\alpha}\pb m, s\big)$. The proof of Proposition \ref{smallversion} is
now complete. \end{proof}

Finally we explain what has to be changed if $c\neq 1$. Lemma \ref{onetomany} is valid without any significant changes in the proof, if we just replace every appearance of ``$(m,p,1)$-system'' with ``$(m,p,c)$-system''. The same holds true for the inductive step in the proof of Proposition \ref{smallversion} up to the point where $s$ is chosen. In order to achieve that $S_U$ contains an $(m+1,p,c)$-system, we would need that $s$ is divisible by $c$ but at this point it is not obvious why this should be the case. Thus we end up with $S_U$ containing an $(m+1,p,c)$-system generated by $\big(s\pb 0_\alpha,\ldots,s\pb m_\alpha,s\big)=\big(t\pb 0,\ldots,t\pb {m+1}\big)$  which is flawed in the sense that $t\pb {m+1}$ is multiplied by $1$ instead of $c$. However we can apply Lemma \ref{onetomany} to see that $S_U$ actually contains such structures generated by $(m+2)$-tuples which form an  $IP$-system $\big(t\pb0_\alpha, \ldots,t\pb{m+1}_\alpha\big)_{\alpha\in\F}$.  Hence we can apply Lemma \ref{divisible} to get that $t\pb{m+1}_\alpha$ is divisible by $c$ if $\alpha\in \F$ is properly chosen. Therefore $S_U$ contains the $(m+1,p,c)$-system generated by $\big(t\pb0_\alpha, \ldots, t\pb m_\alpha,t\pb{m+1}_\alpha/c\big)$ and thus  the $c\neq 1$ version of Proposition \ref{smallversion} also holds for $m+1$.

\section{A proof via ultrafilters}\label{ultrafilterSection}

In this section we use the algebraic structure of the Stone-\v Cech compactification  $\beta \N$  of $\N$ to give a proof to Proposition \ref{viaDeuber}. We start with a brief description of the concepts required in this proof, see \cite{Berg03} for a short ``self contained'' or \cite{HiSt98} for an exhaustive treatment of the algebraic structure on $\beta \N$.

Take $\beta \N$ to be the set of all ultrafilters on $\N$. A non empty system of sets $q\subsetneq \PP(\N)$  is called a \emph{filter} if it is closed under forming supersets and finite intersections. It is an \emph{ultrafilter} if it is a filter with the additional property that whenever $C_1\cup\ldots\cup C_n=\N$, some $C_i$ lies in $D$. Using the axiom of choice, it is possible to show that $|\beta \N|=2^{2^{\N}}$ but the  only elements of $\beta \N$ which can be explicitly constructed are the \emph{principle ultrafilters} $q(n)=\{S\subset \N: n\in S\}$ where $n\in \N$. While not being overly exciting, they allow us to view $\N$ as a subset of $\beta \N$ by identifying each $n\in\N$ with $q(n)\in \beta\N$.   
Using standard properties  of the Stone-\v Cech compactification one obtains that there exists a unique extension of the addition on $\N$ to $\beta \N$ such that the map $q\mapsto r+q$ is continuous for every $r\in \beta \N$ and $q\mapsto q+n$ is continuous for every $n\in \N$. (Note that it is not possible to have both $q\mapsto q+r$ \emph{and} $ q\mapsto r+q$ continuous for all $r\in \beta \N$. In particular $+$ is highly non commutative on $\beta \N$.)
An explicit description of the addition on $\beta \N$ is given by 
\begin{align}\label{ED}
S\in q+r\ \Leftrightarrow \ \{n\in\N:S-n\in q\}\in r.
\end{align}
 (Here $S-n = \{k\in \N:k+n\in \N\}$.) Another interpretation of $+$ is obtained if we interpret ultrafilters as $\{0,1\}$-valued finitely additive measures. Then the addition turns out to be just the convolution of measures. Thus it is not very surprising that $+$ is associative. In fact, the compactness of $\beta \N$ together with continuity of $r+q$ in the right argument guarantees that $(\beta \N, +)$ is a semigroup with quite rich algebraic structure. Moreover, algebraic properties of ultrafilters are nicely linked with combinatorial properties of their elements as is exemplified by the following facts.
\begin{itemize}

\item Similar to finite semigroups, 
$\beta\N$ contains \emph{idempotents}, that is elements $q$ such that $q+q=q$. A set $S\subset \N$ is contained in  an idempotent  ultrafilter iff there exists an $IP$-system $(s_\alpha)_{\alpha \in \F}$ such that  all $s_\alpha$ lie in $S$. 

\item A subset $I$ of a semigroup $(G,+)$ is a two sided ideal if $I+G,G+I\subset I$. 
It can be shown that $\beta\N$ has a smallest two sided ideal $K(\beta \N)$ (with respect to inclusion). A set $S\subset\N$ is piecewise syndetic iff there exists $q\in K(\beta \N)$ such that $S\in q$.   
\end{itemize}

An ultrafilter which is idempotent and lies in the smallest ideal of $\beta \N$ is called a \emph{minimal idempotent}. It was established in \cite[Corollary 6.12]{BeHi90} that $S\subset \N$ is central iff there exists a minimal idempotent $q$ such that $S\in q$.

Replacing piecewise syndetic with positive upper Banach density leads to the  class of  \emph{essential idempotents}:  $q\in\beta \N$ is an essential idempotent iff it is an  idempotent ultrafilter, all of whose elements have positive upper Banach density.   
By \cite[Theorem 2.8]{BeDo08}, $S  \subset\N$ is a $D$-set iff it is contained in some essential idempotent. 

We are now ready to employ these concepts to prove that $D$-sets satisfy the conclusion of Furstenberg's Central Sets Theorem (\cite[Proposition 8.21]{Furs81}).

\begin{theorem}\footnote{While stronger versions of the Central Sets Theorem hold true (see in particular \cite{DeHS08}), we chose to go with this version to keep the formulation simple.}\label{DCST}
Assume that $S$ is a $D$-set and  that $\big(s\pb1_\alpha\big)_{\alpha\in\F}, \ldots, \big(s\pb p_\alpha\big)_{\alpha\in\F}$ are $IP$-systems. There exist sequences $a_ 1,a_ 2, \ldots \in \N$ and  $\alpha_1< \alpha_2<\ldots$ in $\F$ such that 
for all $k_1< \ldots< k_l$ and $i\in \nhat p$,
\begin{equation}\label{frich} \big(a_{k_1} + s\pb i_{\alpha_{k_1}}\big)+ \ldots +
\big(a_{k_l} + s\pb i_{\alpha_{k_l}}\big)\in S.\end{equation}
\end{theorem} 

The following standard Lemma (cf.\ \cite[Lemma 4.14]{HiSt98}) nicely simplifies the inductive process used in the proof of Theorem \ref{DCST}. 
\begin{lemma}\label{idtrick}
Let $q$ be an idempotent ultrafilter, let $S\in q$ and set $S^\star=\{n\in S:S-n\in q\}$. Then $S^\star\in q$ and $S^\star-n\in q$ for all $n\in S^\star$. 
\end{lemma}
\begin{proof}
  $S^\star=S\cap \{m\in \N: S-m  \in q\}\in q$ by (\ref{ED}) and since $q$ is closed under finite intersections. Given $n\in S^\star$, we have $S^\star-n=(S-n)\cap\{m\in\N:S-m\in q\}-n=(S-n)\cap\{m\in\N:(S-n)-m\in q\}.$ The first set lies in $q$ since $n\in S^\star$ and the second set lies in $q$ because $S-n\in q$ and we can substitute $S-n$ for $S$ in (\ref{ED}). 
\end{proof}

\begin{proof}[Proof of Theorem \ref{DCST}.]
 Let $q$ be an essential idempotent such that $S\in q$ and define $S^\star$ as in Lemma \ref{idtrick}. We will inductively construct $a_ 1,a_ 2, \ldots$ and $\alpha_ 1< \alpha_ 2< \ldots\in \F$ such that (\ref{frich}) is satisfied. To keep the induction going we will in fact demand that (\ref{frich}) is even
true with $S$ replaced by $S^\star$. To start the construction use the
fact that $S^\star$ has positive upper Banach density together with  Theorem \ref{combSze} to find $a_ 1$ and $\alpha_1$ such that $a_1+s\pb i_{\alpha_1}\in S^\star$ for all $i\in \nhat p$.

Assume that after  $n$ steps we have found $a_1,\ldots, a_ n$ and $\alpha_1,\ldots ,\alpha_n$ such that all $t$ which are of the form
$$t=\big(a_{k_1}+s\pb i_{\alpha_{k_1}}\big)+\ldots+\big(a_{k_l}+s\pb i_{\alpha_{k_l}}\big)$$ for some $k_1<\ldots<k_l\leq n$ and $i\in \nhat p$ lie in $S^\star$. Then all
sets $S^\star-t$ are in $q$ and hence so is the intersection $B$ of
$S^\star$ with all the sets $ S^\star-t$. Thus we may use Theorem \ref{combSze}
to find $a_{n+1}$ and $\alpha_{n+1}>\alpha_n$\footnote{To see that one can in fact require that $\alpha_{n+1} >\alpha_n$, apply Theorem \ref{combSze} to the $IP$-systems generated by the numbers the sequences $(s\pb i_k)_{k>\max \alpha_n}$. } such that
$a_{n+1}+s\pb i_{\alpha_{n+1}}\in B$ for all $i\in\nhat p$. Then by the definition of $B$, $$ \big(a_{n+1} +s\pb i_{\alpha_{n+1}}\big), t+ \big(a_{n+1} +s\pb i_{\alpha_{n+1}}\big) \in S^\star,$$ for all $t$ as above and for all $i\in \nhat p$. Continuing in this fashion we arrive at the desired statement. \end{proof}

Finally Proposition \ref{viaDeuber} follows from Theorem \ref{DCST} using the following purely combinatorial fact.
\begin{prop}\label{OldWine}
Let  $S\subseteq \N$. Assume that for every $q\in \N$ and  $IP$-systems $\big(s\pb1_\alpha\big)_{\alpha\in\F}, \ldots,$ $ \big(s\pb q_\alpha\big)_{\alpha\in\F}$ there exist sequences $a_ 1,a_ 2, \ldots \in \N$ and  $\alpha_1< \alpha_2<\ldots$ in $\F$ such that 
for all $k_1< \ldots< k_l$ and $i\in \nhat q$,
\begin{equation}\label{frich2} \big(a_{k_1} + s\pb i_{\alpha_{k_1}}\big)+ \ldots +
\big(a_{k_l} + s\pb i_{\alpha_{k_l}}\big)\in S.\end{equation}
(In short, let $S$ be a set which satisfies the conclusion of the Central Sets Theorem, i.e.\ the conclusion of Theorem \ref{DCST} above.) 

Then $S$ contains an $(m,p,c)$-system for all positive integers $m,p,c$.
\end{prop}
The proof of Proposition \ref{OldWine} is sketched in
\cite[page 174]{Furs81} and fully carried out in \cite[Theorem 15.5]{HiSt98}. Therefore we refrain from giving a full proof, but try to explain the required ideas in the case $c=1$.

\begin{proof} To carry out an inductive argument one proves a stronger statement already familiar from Lemma \ref{onetomany}. Fix $S\subset \N$ and $p\in\N$. We show that for each $m\geq 0$ there exists an $IP$-system $\big(s\pb 0_\alpha,\ldots,s\pb m_\alpha\big)_{\alpha \in \F}$ in $\N^{m+1}$ such that
for all $\alpha\in\F$ and all integers $i \pob 0, \ldots, i \pob {m-1}\in [-p,p]^{m+1}$
\begin{equation}\label{Dsys}
s\pb 0_\alpha\in S,\ \  s\pb1_\alpha+i \pob 0 s\pb 0_\alpha\in S,\ \ \ldots,\ \  s\pb m_\alpha+ i \pob {m-1} s\pb{m-1}_\alpha+\ldots + i \pob 0 s\pb0_\alpha\in S.
\end{equation}
The case $m=0$ of our claim asserts precisely that $S$ contains some $IP$-system. This is quite obvious by the assumption on the set $S$. Applying it to the trivial system consisting only of 0's, we find that there exists a sequence $a_1,a_2,\ldots \in \N$ such that $a_{k_1}+\ldots + a_{k_l}\in S$ for all $k_1< \ldots< k_l\in \N$. Setting $s\pb 0_n=a_n $ for $n\in \N$, this means that 
\begin{align}\label{ZeroCase}
\big(s_\alpha\pb 0\big)\in S
\end{align} for all $\alpha \in \F$. 

In order to prove the first non-trivial instance $m=1$, we apply our assumption on $S$ to the $q=(2p+1)$ $IP$-systems 
$$ \big(i \pob 0 s\pb 0_ \alpha \big)_{\alpha\in \F}, \quad (|i \pob 0|\leq p)$$ 
to
find $a_ n, {n\in\N}$ and $\alpha_1< \alpha_2<\ldots $ such that for all $k_1<\ldots<k_l$
\begin{align}\label{FirstCase}
\big(a_{k_1}+i \pob 0 s\pb 0_{\alpha_{k_1}}\big)+ \ldots + \big(a_{k_l}+ i \pob 0 s\pb 0_{\alpha_{k_l}}\big)\in
S.
\end{align}
 Set $t_ n\pb1 = a_ n$ and $t\pb0_ n=s\pb0_{\alpha_n}$ for $n\in\N$. Then it is special case of (\ref{ZeroCase}) that $t\pb 0_\alpha \in S$  for $\alpha\in \F$ and it follows from (\ref{FirstCase}) that $t\pb1_\alpha+i \pob 0 t\pb 0_\alpha\in S$ for $\alpha \in \F$ and all integers  $i_0\in [-p,p]$. Hence the $IP$-system $\big(t_{\alpha}\pb 0, t_\alpha\pb 1 \big)$ witnesses that the case $m=1$ of (\ref{Dsys}) is valid.

To prove the case $m=2$, apply the assumption on $S$ to the $q=(2p+1)^2$ $IP$-systems 
$$ \big(i \pob 1 t\pb 1_ \alpha+i \pob 0 t\pb 0_ \alpha \big)_{\alpha\in \F}, \quad (|i \pob 0|,|i \pob 1|\leq p).$$ The induction continues in the natural way.
\end{proof}

\begin{ack} The authors thank Neil Hindman for useful remarks on an earlier version of this paper. 
\end{ack}

\def\ocirc#1{\ifmmode\setbox0=\hbox{$#1$}\dimen0=\ht0 \advance\dimen0
  by1pt\rlap{\hbox to\wd0{\hss\raise\dimen0
  \hbox{\hskip.2em$\scriptscriptstyle\circ$}\hss}}#1\else {\accent"17 #1}\fi}

\end{document}